\documentclass[a4paper, 11pt]{article}

\usepackage{amsmath}
\usepackage{amssymb}
\usepackage{dsfont}
\usepackage{calc}
\usepackage{eufrak} 
\usepackage{amsthm}
\usepackage{authblk}

\begin{document}

\theoremstyle{definition}
\newtheorem{definition}{\sc Definition}
\theoremstyle{plain}
\newtheorem{lemma}{\sc Lemma}
\newtheorem{theorem}{\sc Theorem}
\newtheorem{proposition}{\sc Proposition}
\newtheorem{corollary}{\sc Corollary}
\newtheorem*{problem}{\sc Problem}
\theoremstyle{remark}
\newtheorem{remark}{Remark}
\newtheorem{example}{Example}
\newtheorem{property}{Property}


\title{On Order-Preserving and Verbal Embeddings\\ of the Group $\mathbb{Q}$}

\author{Arman Darbinyan \quad  Vahagn H. Mikaelian}

\date{}

\maketitle

\begin{abstract}
\noindent
We show that there is an order-preserving embedding of the additive group of rational numbers $\mathbb{Q}$ into a 2-generator group $G$. The group $G$ can be chosen to be a solvable group $G$ of length 3, which is a minimal result in the sense that it cannot be chosen to be neither solvable of length 2, nor a nilpotent group. For any non-trivial word set $V \subseteq F_\infty$  there is an order-preserving verbal embedding of  $\mathbb{Q}$ into a 2-generator group $G$. The embeddings constructed are subnormal.
\end{abstract}

\let\thefootnote\relax\footnote{
{\bf  2010 Math.\ Subject Classification}. 20D35, 20E05, 20E10, 20E22, 20F18, 20F60.

\par
}

\section{Introduction}
\label{section introduction}

The aim of this note is to add some additional properties to the explicit embedding of the additive group of rational numbers $\mathbb{Q}$ into a 2-generator group $G$ we constructed in~$\cite{mikaelian05}$. Existence of such an explicit embedding of $\mathbb{Q}$ was asked by de la Harpe and Bridson in Problem 14.10 (b) of the Kourovka Notebook~$\cite{kourovka}$, and the construction of~$\cite{mikaelian05}$ was to give a positive answer to this question.

It is natural to ask which properties of $\mathbb{Q}$ can be ``inherited'' by $G$, and which additional options can the embedding have. One of most natural properties characterizing $\mathbb{Q}$ is the linear order of rational numbers and, thus, it is reasonable to ask if the group $G$ can be ordered in such a way that its order continues the natural order of rational numbers in the isomorphic image of $\mathbb{Q}$ in $G$.
Clearly, here we are interested in such a linear order on $G$ which is adjusted with the multiplication of the group $G$, that is, is a full order relation in the sense of~\cite{neumann-order, LeviOnOrderedGroups1, LeviOnOrderedGroups2} (see definitions below).

The next option we add to the embedding is \textit{verbality}. For the non-trivial word set $V \subseteq F_\infty$ the embedding of the group $H$ into the group $G$ is said to be $V$-verbal if the isomorphic image of $H$ under this embedding lies in the verbal subgroup $V(G)$ (see definition below). Thus, the concept of  $V$-verbal embedding, suggested by Heineken~\cite{Heineken, HeinekenMikaelianOnnVEmb, Heineken normal},  is the wide generalization of the notion of embedding into the commutator subgroup, embedding into the member of the lower central series, embedding of the $n$'th derived subgroup, etc.. See also recent research on verbal embeddings in~\cite{SubnormalEmbeddingTheorems}-\cite{dissertation_survey}.

Another property of the embedding we deal with is how ``economical'' it can be in the following sense. The 2-generator group $G$, into which $\mathbb{Q}$ is embedded, is a solvable group of length $3$. We show that $G$ cannot be replaced neither by a finitely generated nilpotent group of any class, nor by a finitely generated solvable group of length 2 (that is, by any finitely generated metabelian group). The latter fact continues Neumann's example of Lemma 5.3 in~\cite{NeumannHNeumann}: the  quasi-cyclic group $Z(p^\infty)$ is an example of an infinitely generated abelian group, which cannot be embedded into a finitely generated metabelian group.

\vskip4mm

Since the notions used in this paper not always have standard definitions or notations in the literature, we bring here a brief list of definitions and references to the sources, where the detailed information and main properties can be found.

The group $G$ is \emph{fully ordered}, if a linear order relation $<$  is defined on $G$, such that for arbitrary elements $g_1, g_2 \in G$, $g_1 < g_2$ implies $ g_1 x < g_2 x$ and $x g_1 < x g_2$ for any $x \in G$.  In the literature these days it is familiar to call the fully ordered groups ``linearly ordered groups'', but we use the older terminology because here we  use   a few other linear order relations also, and some difference in terminology is useful.
An embedding $\varphi$ of the fully ordered group $H$ into the fully ordered group $G$ is said to be order-preserving, if for any  $h_1, h_2 \in H$ the relation $h_1 < h_2$ holds in $H$ if and only if $\varphi(h_1) < \varphi(h_2)$ holds in $G$.
For more information about fully ordered groups see ~\cite{neumann-order, LeviOnOrderedGroups1, LeviOnOrderedGroups2}.

\emph{Cartesian} wreath product $A \mbox{~Wr~} B$ of two groups $A$ and $B$ is the semidirect product $B \rightthreetimes A^B$,  where $A^B$ is the set of maps from $B$ to $A$, and $B$ acts on $A^B$ by the rule: for arbitrary $f: B \rightarrow A$, $b\in B$; $f^b(b_0)=f(b_0 b^{-1})$ for any $b_0 \in B$.
For more information about wreath products we refer to $\cite{neumann-varieties}$.

For a group $G$ and for a word set $V \subseteq F_\infty$ the $V$-verbal subgroup $V(G)$ of $G$ is the subgroup generated by all substitutions $v(g_1, \ldots, g_n) \in G$ for all words $v(x_1, \ldots, x_n) \in V$ and for all elements $g_1, \ldots, g_n \in G$. When it is clear from context which word set $V$ is assumed, we call the $V$-verbal subgroup a verbal subgroup. The variety corresponding to $V$ is the variety generated by the factor group $F_\infty / V(F_\infty)$ or, in other words, the set of all groups for which any $v(x_1, \ldots, x_n) \in V$ is an identity. A word set is said to be non-trivial if $V(F_\infty) \not= \{ 1 \}$.
For more information about word sets, varieties, identities see $\cite{neumann-varieties}$.

An embedding $\varphi$ of the group $H$ into the group $G$ is said to be subnormal, if $\varphi(H)$ is a subnormal subgroup in $G$, that is, if there is a finite series of subgroups $\varphi(H) = G_0 \le G_1 \le \dots \le G_n =G$, such that $G_{i-1}$ is a normal subgroup in $G_{i}$ for any $i = 1, \ldots , n$.


\section{The order-preserving explicit embedding construction}
\label{section first embedding}

The following is the main theorem of the current paper:

\begin{theorem}
\label{th. first embedding}
There is an order-preserving subnormal embedding of the additive group of rational numbers $\mathbb{Q}$ into a fully ordered 2-generator group $G$. The group $G$ can be chosen to be a solvable group $G$ of length 3.
\end{theorem}

Before giving the description of the embedding of $\mathbb{Q}$ with the mentioned properties, let us state a lemma 
which we will use later on.

\begin{lemma}
\label{lemma main}
Let $A$ and $B$ be fully ordered groups and $ X \leq A \mbox{~Wr~} B $.
If for each $b\varphi\in X$, $supp(\varphi)$ is well-ordered, then
$X$ can be fully ordered.

\end{lemma}

\begin{proof}
If $b_{1}\varphi_{1}$ and $b_{2}\varphi_{2}$ belong to $X$, then $b_1\varphi_{1} < b_2\varphi_{2}$
is defined as: $  b_1 < b_2 $ or $ b_1 = b_2$  and $\phi_1(b) < \phi_2(b)$ for the least $b$ for which $\phi_1(b) \neq \phi_2(b)$.

Let $U_1$ be the subset of the union $supp(\phi_1) \cup supp(\phi_2)$ consisting of elements $b$ for which $\phi_1(b)<\phi_2(b)$.
Since $supp(\phi_1)$ and $supp(\phi_2)$ are well-ordered, $U_1$ has a least element $u_1$. Similarly define $U_2$ of elements
$U_2$ of elements $\phi_2(b)< \phi_1(b)$ and take its least element $u_2$. Since $B$ is fully ordered, its order is linear and either $u_1 < u_2$, in which case $\phi_1 < \phi_2$, or $u_2<u_1$, in which case $\phi_2 < \phi_1$. The only exception is the case when $\phi_1 = \phi_2$, of course. Thus, the order $<$ is linear on $X$.

We must show that this order is also full order. Let us take an arbitrary element $b_3 \varphi_3 \in X$ and show that
if $b_1\varphi_{1} < b_2\varphi_{2}$, then $b_1 b_3 \varphi_1^{b_3} \varphi_3 < b_2 b_3 \varphi_2^{b_3}\varphi_3$ and $b_3 b_1 \varphi_3^{b_1}\varphi_1<b_3 b_2 \varphi_3^{b_2}\varphi_2$\\

If $b_1<b_2$ then this is obvious. So we need to consider the case when $b_1=b_2, \varphi_1<\varphi_2$. Hence it is sufficient to show that
\begin{equation}
\label{astxanish}
\tag{*}
\varphi_1^{b_3} \varphi_3 < \varphi_2^{b_3}\varphi_3
\end{equation}
and
\begin{equation}
\label{yerku astxanish}
\tag{**}
\varphi_3^{b_1} \varphi_1 < \varphi_3^{b_1}\varphi_2.
\end{equation}
We have:
\begin{align*}
	&\min\{b\in B \mid \varphi_1^{b_3}(b) \varphi_3(b) < \varphi_2^{b_3}(b)\varphi_3(b) \} \\
	&\quad = \min \{b\in B \mid \varphi_1(b b_3^{-1}) < \varphi_2(b b_3^{-1}) \} \\
	&\quad = \min \{b\in B \mid \varphi_1(b) < \varphi_2(b) \}b_3 \\
	&\quad < \min \{b\in B \mid \varphi_1(b) > \varphi_2(b) \}b_3\\
	&\quad = \min\{b\in B \mid \varphi_1^{b_3}(b) \varphi_3(b) \\ 
	&\quad > \varphi_2^{b_3}(b)\varphi_3(b),
\end{align*}
hence (\ref{astxanish}) is proved. And also:
\begin{align*}
	&\min\{b\in B \mid \varphi_3(b) \varphi_1(b) < \varphi_3^{b b_1^{-1}}(b)\varphi_2(b) \} \\
	&\quad = \min \{b\in B \mid \varphi_1(b) < \varphi_2(b) \} \\
	&\quad > \min \{b\in B \mid \varphi_1(b) > \varphi_2(b) \} \\
	&\quad = \min\{b\in B \mid \varphi_3(b) \varphi_1(b) \\
	&\quad > \varphi_3^{b b_1^{-1}}(b)\varphi_2(b),
\end{align*}
hence (\ref{yerku astxanish}) is also proved.
So Lemma 1 is proved.
\end{proof}

\begin{remark}
In the proof of the Lemma $\ref{lemma main}$ we not only proved that the subgroup $X$ can be fully ordered, but also suggested an explicit group order.
Also, if the intersection of $X$ and the first copy of $A$ in $A \mbox{Wr}B$ contains $a_1, a_2 \in A$, then $a_1<a_2$
 in $A$ $\Rightarrow$ $a_1 < a_2$ in $X$.
\end{remark}

\begin{remark}
The requirement of being well-ordered of support in the statement of Lemma $\ref{lemma main}$ is necessary. Moreover, it is true that the
wreath product of two infinite groups can not be fully ordered $\cite{neumann-order}$.
\end{remark}

~\\

Bellow we have constructed an embedding of $\mathbb{Q}$ into a 2-generator subgroup of $ W = (\mathbb{Q} \mbox{~Wr~} C)\mbox{~Wr~} Z$, where
$ C=\langle c \rangle$ is an infinite cyclic group generated by $c$, and $ Z= \langle z \rangle$ is an infinite cyclic group generated by $z$.

Firstly, for each positive integer  $n$, choose in the base subgroup $\mathbb{Q}^{C}$ of the Cartesian wreath product $\mathbb{Q} \mbox{~Wr~} C$
the elements $\varphi_n$ and $\tau_n$:\\
\[  
\varphi_n (c^{i})= \left\{
                      \begin{array}{ll}
                       \frac{1}{n} & \mbox{if $i=0$ ,}\\
                        0 & \mbox{if $i \neq 0 $ ,}
                      \end{array}
                    \right.
\tau_n (c^{i})= \left\{
                      \begin{array}{ll}
                       0 & \mbox{if $i<0$},\\
                       \frac{-1}{n} & \mbox{if $i \geq 0 $ }.
                      \end{array}
                    \right.
\] 
~\\
The reason of such selection is in the following relations:\\
\begin{equation}
[\tau_n, c]= \varphi_n, [\tau_m, \tau_n]=1 \mbox{ for any n, k $>$ 0}.
\end{equation}
~\\
The first of the relations (1) trivially follows from
\[ 
[\tau_n, c](c^{i}) = \tau_n^{-1}\tau_n^{c}(c^i)= \left\{
                      \begin{array}{ll}
                       -0+0=0 & \mbox{if $i<0$},\\
                       \frac{1}{n}+0=\frac{1}{n} & \mbox{if $i = 0 $ },\\
                       \frac{1}{n}+\frac{-1}{n}=0 & \mbox{if $ i>0$}.
                      \end{array}
                    \right.
\] 
~\\
The second of the relations (1) trivially follows from the fact that $\mathbb{Q}^{C}$ is abelian.

~\\
In the base subgroup $(\mathbb{Q} \mbox{~Wr~} C)^Z$ of $W$, take an element $\alpha$ defined as

\[ 
\alpha(z^{j})=\left\{
                      \begin{array}{ll}
                       1=1_{\mathbb{Q}\mbox{~Wr~} C} & \mbox{if $j<0$},\\
                       c & \mbox{if $j = 0 $ },\\
                       \tau_j & \mbox{if $ j>0$}.
                      \end{array}
                    \right.
\]
~\\
Put $G=\langle \alpha, z \rangle$ and define the embedding $\Phi: \mathbb{Q}\rightarrow G$ as
\[ 
\Phi: \frac{m}{n} \mapsto [z^n \alpha z^{-n}, \alpha]^m = [\alpha^{z^{-n}}, \alpha]^m ~\mbox{for any}~ \frac{m}{n} \in \mathbb{Q}, n>0.
\] 
~\\
That $\Phi$ is a homomorphism and an injection could be checked directly. But to avoid very long calculations, we consider the structure of the
commutator $[\alpha^{z^{-n}},\alpha]$ first:

\[ 
[\alpha^{z^{-n}}, \alpha](z^j)= [\alpha(z^{j+n}), \alpha(z^j)]=
                  \left\{
                      \begin{array}{ll}
                       [1,1]=1 & \mbox{if $j<-n$},\\
                       {}[c,1]=1 & \mbox{if $j = -n $ },\\
                       {}[\tau_{j+n},1]=1 & \mbox{if $ -n<j<0$},\\
                       {}[\tau_{n},c]=\varphi_n & \mbox{if $ j=0$},\\
                       {}[\tau_{j+n}, \tau_{j}]=1 & \mbox{if $j>0$}.
                      \end{array}
                    \right.
\] 
~\\
This means that $[\alpha^{z^{-n}},\alpha]$ is nothing else but the image $\phi_{n}^{*}$ of the coordinate element $\varphi_n$ in
the ``the first copy" of the group $\mathbb{Q} \mbox{Wr} C$ in $W$:

\[ 
\varphi_{n}^{*}(z^j)=
 \left\{
                      \begin{array}{ll}
                       \varphi_n & \mbox{if $j=0$},\\
                       1=1_{\mathbb{Q} \mbox{~Wr~} C} & \mbox{if $j \neq 0 $ }.
                      \end{array}
                    \right.
\] 
~\\

Therefore the elements
\[
\Phi(\mathbb{Q})=\{\Phi(\frac{m}{n})=(\varphi_n^*)^m) \mid \frac{m}{n} \in \mathbb{Q}, n > 0 \}
\] 
~\\
do form a subgroup isomorphic to $\mathbb{Q}$ in $G$, and the mapping $\Phi$ is injective. Finally, it easily follows from the equalities
\[ 
\Phi(\frac{1}{n})\Phi(\frac{1}{n^{\prime}})=\varphi^*_n\varphi^*_{n^{\prime}}, \Phi(\frac{m}{n})=(\varphi_n^*)^m=\underbrace{\varphi_n^*+\ldots+\varphi_n^*}_{n ~times}
\\
\] 
that $\Phi$ is a homomorphism.

Each element of $G$ can be presented in the form
\begin{equation}
\label{Yerku}
z^k(\alpha^{z^{k_1}})^{n^1}\cdot \ldots \cdot (\alpha^{z^{k_s}})^{n^s},
\end{equation}
which follows from the fact that
\[
\alpha^n z^m=z^m(\alpha^{z^m})^n.
\]
Let us denote the product of $\alpha$-factors of (2) by $\tilde{\alpha}$.

\begin{equation}
\tilde{\alpha}(z^i) = 1_{\mathbb{Q}\mbox{~Wr~} C}, \mbox{~if $i<\min \{k_1, \ldots ,k_s \}$}.
\end{equation}
It is clear that $G \leq T \mbox{~Wr~} \mathbb{Z}$, where $T=\langle c, \tau_i \mid i\in \mathbb{Z}, i>0 \rangle$ and $supp(\tilde{\alpha})$ is well-ordered.

Hence, by Lemma \ref{lemma main}, in order to show that $G$ is fully ordered with order relation defined in the proof of Lemma \ref{lemma main}, we only need to show that $T$ is fully ordered.\\

Similarly to~\ref{Yerku} each element of $T$ can be presented in the form:
\[
c^{k_0}(\tau_{i_1}^{c^{k_1}})^{n_1}(\tau_{i_2}^{c^{k_2}})^{n_2} \cdot \ldots \cdot (\tau_{i_m}^{c^{k_m}})^{n_m} \mbox{~ where~} k_1 \leq k_2 \leq \ldots \leq k_m,
\]
which follows from $(\tau_{i}^{c^{j}})^n c^k = c^k (\tau_{i}^{c^{k+j}})^n$ and from the fact that $(\tau_{i_p}^{c^{k_p}})^{n_p}$ commutes with $(\tau_{i_q}^{c^{k_q}})^{n_q}$, where $1 \leq p,q \leq m$.
Obviously $T$ is a subgroup of $\mathbb{Q} \mbox{~Wr~} C$,\\

If we denote the product of $\tau$-factors of (3) by $\tilde{\beta}$, then
\begin{equation}
\tilde{\beta}(c^i)
=\left\{
                      \begin{array}{ll}
                       0 & \mbox{~if $i<k_1$},\\
                       \frac{n_1}{i_1}+\frac{n_2}{i_2}+\ldots + \frac{n_s}{i_s} & \mbox{~if $k_s >i \geq k_{s-1} $ }.
                      \end{array}
                    \right.
\end{equation}
~\\
Note that we not only proved that $G$ is fully ordered, but also suggested an explicit order relation. Indeed, we firstly embedded
$\mathbb{Q}$ into a countably generated subgroup $T$ of $\mathbb{Q} \mbox{~Wr~} C$, then embedded $T$ into a $2$-generator subgroup $G$ of
$ (\mathbb{Q} \mbox{~Wr~} C) \mbox{~Wr~} Z $ and it follows from (3) and Lemma~$\ref{lemma main}$ that the first embedding preserves full order, and it follows from (4) and Lemma~$\ref{lemma main}$ that the second embedding also preserves full order.

The embedding $\Phi$ of $\mathbb{Q}$ into $G$ described above is subnormal. As a subgroup of an abelian group $\langle \phi_n \mid n \in \mathbb{N} \rangle $ is normal in the first copy of $\mathbb{Q}$ in $ \mathbb{Q} \mbox{~Wr~} C$, the first copy is normal in $\mathbb{Q}^C$, and $\mathbb{Q}^C$ is a normal subgroup of $\mathbb{Q} \mbox{~Wr~} C$. Hence $\mathbb{Q}^{\Phi}= \langle \phi_n^* \mid n \in \mathbb{N} \rangle$ is subnormal in the first copy of $\mathbb{Q} \mbox{~ Wr ~} C $ in $(\mathbb{Q} \mbox{~ Wr ~} C ) \mbox {~ Wr ~} Z$. Therefore, $\mathbb{Q}^{\Phi}$ is subnormal in $G$, because the first copy of $\mathbb{Q} \mbox{~ Wr ~} C$ in $(\mathbb{Q} \mbox{~ Wr ~} C ) \mbox {~ Wr ~} Z$ is subnormal subgroup of $(\mathbb{Q} \mbox{~ Wr ~} C ) \mbox {~ Wr ~} Z$.

And $G$ is a solvable group of length at most 3 since  $(\mathbb{Q} \mbox{ Wr } C ) \mbox { Wr } Z$ is a solvable group of length 3.

This concludes the proof of Theorem~\ref{th. first embedding}.

~\\
\section{Additional properties of the embedding}
\label{section additional properties}
~\\
Now we will examine additional properties of $\mathbb{Q} $ and $G$.

\begin{property}         
$G$ is torsion free.
\end{property}

\begin{proof}
It follows from the fact that $(\mathbb{Q} \mbox{~Wr~} C) \mbox{~Wr~} Z$ is torsion-free group.
\end{proof}
~\\
Product of two group varieties defined as a variety consisting of all extensions of a group from the first variety by a group of the second variety (see $\cite{neumann-varieties}$). As we know wreath product $A \mbox{~Wr~}B$ of two groups $A$ and $B$ is an extension of the Cartesian product $A^{B}$ by $B$. Hence, the following property is true.

\begin{property}
 $G$ belongs to the variety $\mathfrak{AAA} = \mathfrak{S}_3$, since $\mathbb{Q}$, $C$ and $Z$ are abelian  (by $\mathfrak{S}_3$ we denoted the variety of solvable groups of length 3). This means that $G$ is a solvable group of length 3, as we noted in proof of Theorem~\ref{th. first embedding}. The propositions below show that here $\mathfrak{S}_3$ cannot be replaced by $\mathfrak{S}_2$.
\end{property}

The following propositions show that the embedding of $\mathbb{Q}$ in $G$ is the most economical, in the sense that $\mathbb{Q}$ cannot be embedded into nilpotent group, or in a metabelian group.

\begin{proposition}
$\mathbb{Q} $ cannot be embedded into a finitely generated nilpotent group.
\end{proposition}

\begin{proof}
 The fact follows from P.Hall's result~\cite{hall-max}, which states that every finitely generated nilpotent group satisfies the maximal condition for subgroups. Thus every subgroup of such a group is finitely generated. Since $\mathbb{Q}$ is not finitely generated, it cannot be embedded into finitely generated nilpotent group.
 \end{proof}

\begin{proposition}
$\mathbb{Q}$ cannot be embedded into a finitely generated metabelian group.
\end{proposition}

\begin{proof}
Suppose that the converse is true, that is to say $\mathbb{Q}$ is embedded into a finitely generated metabelian group $G$.
Evidently, a finitely generated abelian group satisfies the maximal condition for subgroups, also by the P.Hall's result~\cite{hall-max}, finitely generated metabelian groups satisfy the maximal condition for normal subgroups. Therefore the commutator $G'$ is finitely generated and hence $G' \cap \mathbb{Q}$ is finitely generated (therefore cyclic).

$G/G'$ is a finitely generated abelian group, therefore the subgroup $\mathbb{Q} G'/ G'$ of $G/G'$ also is finitely generated. But we have $\mathbb{Q} G' / G' \simeq \mathbb{Q} / G' \cap \mathbb{Q}$, which is not finitely generated because $\mathbb{Q}$ is not finitely generated. Contradiction.
\end{proof}


\section{The verbal embedding of $\mathbb{Q}$}
\label{section verbal embedding}

For definitions and basic facts about non-trivial word sets, verbal subgroups and embeddings see Introduction above and literature cited there.

\begin{theorem}
\label{th second embedding}
For any non-trivial set of words $V$ there is an order-preserving subnormal embedding of the additive group of rational numbers $\mathbb{Q}$ into a fully ordered 2-generator  group $G$.
\end{theorem}

For embedding construction purposes we need to find a fully ordered torsion free nilpotent group $S$ with a non-trivial positive element $a \in V(S)$, as it is done in $\cite{mikaelian02}$.  As a such group we take $S= F_k (\mathfrak{N}_c)$, where $c$ is the least integer, such that $\mathfrak{N}_c$ is not contained in the variety defined by $V$ and $k$ is such that $S \notin var(F_{\infty} / V(F_{\infty}))$.  A full order relation can be defined in $S$ (see $\cite{mikaelian02}$, we omit the routine details to much shorten the proof since the exact method of construction of that embedding is immaterial for purposes of this proof).

We take an arbitrary non-trivial element $a \in V(S) \neq {1}$. In any case we can assume $a$ to be positive ($1 < a$), for we  always are in position to replace our order relation $<$ by the inverse relation $<^{-1}$ (for details see $\cite{mikaelian02}$).

As an element of $V(S)$ our element $a$ has the presentation
\[
   a= (v_1(a_{11},\ldots, a_{1t_1}))^{\varepsilon_1} \cdots (v_d(a_{d1}, \ldots, a_{dt_d}))^{\varepsilon_d}
\]
where $\varepsilon_i = \pm 1, v_i \in V, a_{ij} \in S (i = 1, \ldots, d; j = 1, \ldots, t_i).$

Now let us consider the Cartesian wreath product $\mathbb{Q} \mbox{~Wr~} S$ and for each positive integer $n$ define an element $\chi_n$ and $\psi_n$ as follows:
$$
\psi_n(s) =  \left\{
                     \begin{array}{ll}
                      \frac{1}{n} & \mbox{if $s=1$,}\\
                       0 & \mbox{otherwise ,}
                      \end{array}
                    \right.
 \chi_n(s) =  \left\{
                      \begin{array}{ll}
                      \frac{1}{n} & \mbox{if $s=a^i, i = 0,1,2,\ldots ,$}\\
                        0 & \mbox{otherwise .}
                      \end{array}
                   \right.
$$
Let us consider the subgroup of $\mathbb{Q} \mbox{~Wr~} S$
\[
    T= \langle \chi_n, a_{ij} \mid n \in \mathbb{N}, i = 1, \ldots, d; j = 1, \ldots, t_i \rangle
\]
\begin{lemma}
\label{lemma verbal embedding}
Let $V$ be an arbitrary non-trivial word set and $T=T(\mathbb{Q},V)$ is that constructed above. Then:\\
~(1) $\mathbb{Q}$ can be embedded in $T$ such that its image lies in $V(T)$.\\
~(2) $T$ can be fully ordered, such that the order of $\mathbb{Q}$ will be preserved by the embedding.
\end{lemma}

\begin{proof}
For the proof of the first part we just need to notice that $ \frac{m}{n} \mapsto \psi_n^m $ is an embedding of $\mathbb{Q}$ into $T$ and $\psi_n = a^{-1}a^{\chi_n} \in V(T)$, because $a \in V(T)$ and $V(T)$ is normal in $T$.

For the proof of the second part let us take an arbitrary element of $T$\\
\begin{equation}
    x_1a_1x_2a_2 \ldots x_ka_k
\end{equation}
where $k$ is some integer, $x_i \in  \langle \chi_n \mid n \in \mathbb{N} \rangle$ $i= 1,\ldots ,k$ and $a_j \in \langle a_{pq} \mid p=1, \ldots, d; q = 1, \ldots t_p \rangle $, $i,j = 1, \ldots, k$. Obviously each element of $T$ can be presented in this form. Also it is obvious that $supp(x_i), i= 1, \ldots , n$ is well-ordered. Now let us transform the presentation (5) to the from:

\begin{equation}
    a_1a_2 \ldots a_k x_1^{a_1a_2 \ldots a_k}x_2^{a_2a_3 \ldots a_k} \ldots x_k
\end{equation}
~\\
It is clear that $supp( x_1^{a_1a_2 \ldots a_k} x_2^{a_2a_3 \ldots a_k} \ldots x_k ) \subseteq \bigcup_{i=1}^k supp(x_i^{ a_i \ldots a_k})$. Now the proof follows from the Lemma $\ref{lemma main}$, by taking into the account the fact that a finite union of well-ordered sets is well-ordered (thanks to the fact that the active group of the wreath product is linearly ordered, since it is fully ordered).
\end{proof}

For the later use let us note the following commutator identities: $[\chi_i, \chi_j]=1$, $[a, \chi_i]=\psi_i$.

~\\

The next step is to embed $T$ into a subgroup of the Cartesian wreath product $ T \mbox{~Wr~} C$, where $C=\langle c \rangle$ is an infinite cyclic group.
Let us denote by $\rho_g$ the element of the first copy of $T$ in base group $T^C$ corresponding to $g \in T$. In addition, define
\[
\pi_g(c^i) =  \left\{
                      \begin{array}{ll}
                       g & \mbox{if $i \geq 0$,}\\
                        1 & \mbox{otherwise .}
                      \end{array}
                    \right.
\]
~\\
Then $[ \pi_{g^{-1}}, c] = \rho_g$ and it is important to note that the first copy of $T$ lies in the derived subgroup of the group
\[
D = \langle \pi_g, c \mid g \in T \rangle.
\]
~\\
So we can embed $T$ into $D$ by the rule $g \mapsto \rho_g $, for all $g \in T$.
Now we should show that $D$ can be fully ordered. Using the same transformation, that we have described above we can present every element of $D$ in the form
\[
    d = c^k \pi_{g_1}^{c^{k_1}}\pi_{g_2}^{c^{k_2}}\cdot \ldots \cdot \pi_{g_n}^{c^{k_n}}
\]

It is easy to see that the support of the ``right part" of $D$ is well ordered, therefore by Lemma \ref{lemma main}, $D$ can be fully ordered. Now it can be checked directly, that the order described in the proof of Lemma \ref{lemma main} is preserved by the above described embedding.

Obviously $D$ is countable. Let us enumerate the elements of $D$, such that
\[
D = \{ d_0, d_1, \ldots, d_n, \ldots ; n\in \mathbb{N} \}.
\]
Define an element $\omega$ in $D^Z$:
\[
\omega(z^i) =  \left\{
                      \begin{array}{ll}
                       d_k & \mbox{if $i = 2^k, k = 0, 1,2, \ldots$,}\\
                        1 & \mbox{otherwise .}
                      \end{array}
                    \right.
\]

For arbitrary $d_n$ (that is, for every $n$) $\omega^{(z^{-2^n})}(1) = d_n$ holds. So for each pair $d_n$ and $d_m$ we have
\[
      [ \omega^{(z^{-2^n})}, \omega^{(z^{-2^m})} ](1) = [d_n, d_m].
\]
Furthermore, for arbitrary $j \neq 0$,
\[
      [ \omega^{(z^{-2^n})}, \omega^{(z^{-2^m})} ](z^j) = 1
\]
(see \cite{mikaelian02}). Thus every element of the derived subgroup $D'$ belongs to the derived subgroup of a $2$-generator group,
\[
  G=\langle \omega , z \rangle.
\]

Now by using the same arguments as it was mentioned above, we can check that $G$ can be fully ordered, such that the embedding of $\mathbb{Q}$ into $G$ preserves the ``natural" order of $\mathbb{Q}$. Indeed, each element of $G$ can be presented in the form
\[
z^k (\omega^{z^{k_1}})^{n^1}\cdot \ldots \cdot (\omega^{z^{k_s}})^{n^s}
\]
and $\left[ (\omega^{z^{k_1}})^{n^1}  \cdots   (\omega^{z^{k_s}})^{n^s} \right](z^i)=1_{T \mbox{~Wr~} C}$, if $i < \min\{k_1, \ldots, k_s\}$.
Hence, supports of the ``right parts" of the elements of $G$ are well ordered, thus by Lemma \ref{lemma main} $G$ can be fully ordered.

Subnormality of the embedding of $\mathbb{Q}$ described in this section can be shown analogously to the similar property of the embedding from the Section \ref{section first embedding}, again by taking into account the fact that the first copy of the passive group is subnormal in wreath product.

This concludes the proof of Theorem~\ref{th second embedding}. \\

Observing the steps of the embedding construction, it is easy to note that the $2$-generator group $G$ belongs to the variety $\mathfrak{A} \mathfrak{N}_c \mathfrak{A} \mathfrak{A} \subseteq \mathfrak{S}_{c+3}$, therefore $G$ is a solvable group of length at most $c+3$.

{\footnotesize
\vskip6mm

Arman Darbinyan

Department of Informatics and Applied Mathematics

Yerevan State University

Alex Manoogian 1

0025 Yerevan, Armenia

Email: arman.darbin@gmail.com

\vskip3mm

Vahagn H. Mikaelian

Department of Informatics and Applied Mathematics

Yerevan State University

Alex Manoogian 1

0025 Yerevan, Armenia

Email: v.mikaelian@gmail.com
}


~\\

\end{document}